\documentclass{amsart}
\usepackage[english]{babel}
\usepackage{amssymb,enumerate,bbm,amsmath}
\usepackage[colorlinks=true,linkcolor=blue,citecolor=blue,urlcolor=blue]{hyper ref}
\numberwithin{equation}{section}

\newtheorem{theo}{Theorem}[section]

\newtheorem{lem}{Lemma}[section]

\newtheorem{rem}{Remark}[section]

\renewcommand{\(}{\left(}
\renewcommand{\)}{\right)}

\renewcommand{\-}{\overline}

\newcommand{\R}{\mathbb{R}}

\renewcommand{\S}{\mathbb{S}}

\renewcommand{\a}{\alpha}
\renewcommand{\b}{\beta}
\newcommand{\g}{\gamma}

\newcommand{\e}{\varepsilon}

\renewcommand{\l}{\lambda}

\renewcommand{\t}{\theta}

\renewcommand{\L}{\Lambda}

\newcommand{\ra}{\rightarrow}

\newcommand{\mrm}{\mathrm}

\newcommand{\Diam}{\mrm{Diam}}

\newcommand{\divv}{\mrm{div}}

\begin{document}

\title[An Eigenvalue Pinching Theorem For Compact Hypersurfaces]{An Eigenvalue Pinching Theorem For Compact Hypersurfaces In A Sphere}
\author{Yingxiang Hu and Hongwei Xu}
\address{Center of Mathematical Sciences \\ Zhejiang University \\ Hangzhou 310027 \\ China\\}
\email{huyx@cms.zju.edu.cn, xuhw@cms.zju.edu.cn}
\date{}
\thanks{Research supported by the National Natural Science Foundation of China, Grant Nos. 11371315, 11201416.}

\begin{abstract}
 In this article, we prove an eigenvalue pinching theorem for the first eigenvalue of the Laplacian on compact hypersurfaces in a sphere. Let $(M^n,g)$ be a closed, connected and oriented Riemannian manifold isometrically immersed by $\phi$ into $\S^{n+1}$. Let $q>n$ and $A>0$ be some real numbers satisfying $|M|^\frac{1}{n}(1+\|B\|_q)\leq A$. Suppose that $\phi(M)\subset B(p_0,R)$, where $p_0$ is a center of gravity of $M$ and radius $R<\frac{\pi}{2}$. We prove that there exists a positive constant $\e$ depending on $q$, $n$, $R$ and $A$ such that if $n(1+\|H\|_\infty^2)-\e\leq \l_1$, then $M$ is diffeomorphic to $\S^n$. Furthermore, $\phi(M)$ is starshaped with respect to $p_0$, Hausdorff close and almost-isometric to the geodesic sphere $S\(p_0,R_0\)$, where $R_0=\arcsin\frac{1}{\sqrt{1+\|H\|_\infty^2}}$.
\end{abstract}

{\maketitle}
\section{Introduction}\label{introduction}
The eigenvalue problem for compact hypersurfaces in a sphere is a subject of particular interest in geometry of submanifolds. Throughout this article, let $(M^n,g)$ be a closed, connected and oriented $n(\geq 2)$-dimensional Riemannian manifold isometrically immersed by $\phi$ into the unit sphere $(\S^{n+1},can)$, i.e., $\phi^\ast can=g$. When $M^n$ is a minimal hypersurface in $\S^{n+1}$, it follows from Takahashi's theorem \cite{Takahashi1966} that the first nonzero eigenvalue $\l_1(M)$ of the Laplacian of $(M^n,g)$ is not greater than $n$. Inspired by this, Yau \cite{Yau1982} conjectured that for every closed embedded minimal hypersurface $M^n$ of $\S^{n+1}$, $\l_1(M)=n$. Up to now, Yau's conjecture is far from being solved. The first breakthrough to this conjecture was made by Choi and Wang \cite{Choi-Wang1983}. They proved that $\l_1(M)\geq\frac{n}{2}$. A careful argument (see \cite{Brendle2012}, Theorem $5.1$) shows that the strict inequality holds, i.e.,  $\l_1(M)>\frac{n}{2}$. More recently, Tang and Yan \cite{Tang-Yan2013} made a new breakthrough and proved that Yau's conjecture is true for closed minimal isoparametric hypersurfaces of $\S^{n+1}$. For further discussions, we refer to \cite{Tang-Xie-Yan2014}.

Besides the minimal case, by considering the canonical embedding of $\S^{n+1}$ into $\R^{n+2}$ and applying the Reilly's inequality \cite{Reilly1977} for the Euclidean submanifolds, one can easily obtain
\begin{align}\label{Reilly-inequality-spherical}
\l_1(M) \leq n(1+\|H\|_2^2).
\end{align}
Here the $L^p$-norm is always normalized, i.e., for any $f\in C^\infty(M)$,
$$\|f\|_p:=\(\frac{1}{|M|}\int_M |f|^p dv\)^\frac{1}{p},$$
where $|M|$, $dv$ are the Riemannian volume and volume element of $(M^n,g)$ respectively. From the H\"older inequality, it follows that for $p\leq q$, we have $\|f\|_p \leq \|f\|_q$. Furthermore, $\lim_{p\ra \infty}\|f\|_p=\|f\|_{\infty}$.

In \cite{Grosjean2002}, Grosjean showed that the equality (\ref{Reilly-inequality-spherical}) holds if and only if $\phi$ is minimal or $\phi(M)$ is a geodesic hypersphere of $\S^{n+1}$. Moreover, if $\phi(M)$ is contained in an open hemisphere, the equality holds if and only if $\phi(M)$ is a geodesic hypersphere of $\S^{n+1}$. This follows from the fact \cite{Myers1951} that there exists no closed minimal hypersurface in an open hemisphere. Motivated by the rigidity result, it is natural to investigate the eigenvalue pinching problem for compact hypersurfaces by relaxing the equality (\ref{Reilly-inequality-spherical}).

We denote the open (resp. closed) geodesic ball of center $p$ and radius $R$ by $B(p,R)$ (resp. $\-B(p,R)$), and the geodesic sphere of center $p$ and radius $R$ by $S(p,R)$. In 2012, Grosjean and Roth \cite{Grosjean-Roth2012} studied the eigenvalue pinching problem for the compact hypersurface of ambient spaces with bounded sectional curvature. When the ambient space is restricted to $\S^{n+1}$, their result (see \cite{Grosjean-Roth2012}, Corollary $1.1$) can be stated as follows.
\begin{theo}\label{Grosjean-Roth-spherical}
Let $(M^n,g)$$(n\geq 2)$ be a closed, connected and oriented Riemannian manifold isometrically immersed by $\phi$ into $\S^{n+1}$. Assume that $|M|\leq c\omega_n$ with $c\in (0,1)$ and $\phi(M)$ lies in a geodesic ball of radius $\frac{\pi}{8}$. Let $p_0$ be the center of mass of $M$. Let $\e<\frac{1}{6}$, $q>n$ and $A>0$ be some real numbers satisfying $\max\(|M|^\frac{1}{n}\|H\|_\infty,|M|^\frac{1}{n}\|B\|_q \)\leq A$. Then there exist positive constants $C=C(q,n,A)$ and $\a=\a(q,n)$ such that if
\begin{align*}
n(1+\|H\|_\infty^2)\leq \l_1(M)(1+\e),
\end{align*}
holds with $C\e^\a<1$, then
\begin{align*}
d_H\(\phi(M),S(p_0,R_0)\)\leq \frac{C\e^\a}{\sqrt{1+\|H\|_\infty^2}},
\end{align*}
where $d_H$ denotes the Hausdorff distance and $R_0=\arcsin\frac{1}{\sqrt{1+\|H\|_\infty^2}}$. Furthermore, $M$ is diffeomorphic to $\S^n$ and almost-isometric to $S(p_0,R_0)$.
\end{theo}

A natural question is: to what extent can one enlarge the geodesic ball in Theorem \ref{Grosjean-Roth-spherical}? In this article, we will prove the following eigenvalue pinching theorem for compact hypersurfaces in a sphere.
\begin{theo}\label{main-theorem}
Let $(M^n, g)$$(n\geq 2)$ be a closed, connected and oriented Riemannian manifold isometrically immersed by $\phi$ into $\S^{n+1}$.
Let $p_0$ be a center of gravity of $M$. Let $q>n$ and $A>0$ be some real numbers satisfying $|M|^\frac{1}{n}(1+\|B\|_q)\leq A$. Suppose that $\phi(M)\subset \-B(p_0,R)$ with $R<\frac{\pi}{2}$ and set $\b=\cot(R)$. There exists an explicit positive constant $\e_0=\e_0(q,n,\b,A)$ such that if
$$
n(1+\|H\|_\infty^2)-\e \leq \l_1(M),
$$
holds with $\e<\e_0$, then $M$ is diffeomorphic to $\S^n$. Moreover,
\begin{enumerate}[(i)]
\item $\phi(M)$ is starshaped with respect to $p_0$;
\item There exists a positive constant $D=D(n,\b,A)$ such that
\begin{align}\label{Hausdorff-closeness}
d_H(\phi(M),S(p_0,R_0)) \leq \frac{D\e^\frac{1}{2(2n+1)}}{\sqrt{1+\|H\|_\infty^2}};
\end{align}
\item $M$ is almost-isometric to $S\(p_0,R_0\)$. More precisely, there exist positive constants $D_1=D_1(n,\b,A)$ and $D_2=D_2(q,n,\b,A)$ and a diffeomorphism $F:(M,d_1)\ra \(S\(p_0,R_0\),d_2\)$ such that
\begin{align}\label{almost-isometry}
\left| d_2(F(x_1),F(x_2))-d_1(x_1,x_2) \right| \leq \frac{\max\left\{ D_1\e^\frac{1}{2(2n+1)},D_2\e^\frac{q-n}{q-n+qn}\right\}}{\sqrt{1+\|H\|_\infty^2}},
\end{align}
for any $x_1,x_2\in M$, and $d_1,d_2$ are the standard distance functions on $M$ and $S\(p_0,R_0\)$ respectively.
\end{enumerate}
\end{theo}
\begin{rem}
It should be mentioned that our result is proved without the assumption on volume of hypersurfaces. In Section \ref{example}, we construct an example to show that the condition on radius of geodesic ball is optimal. For other pinching results on compact hypersurfaces, we refer readers to \cite{Aubry-Grosjean-Roth2010,Colbois-Grosjean,Huxuzhao2015,Roth2007,Roth2008,Roth2008-agag}.
\end{rem}

\section{Notations and Lemmas}\label{notations-and-lemmas}
First of all, we recall that the {\em Hausdorff distance} between two compact subsets $A$ and $B$ of a metric space $(M,d)$ is given by
\begin{align*}
d_H(A,B)=\inf\left\{\eta>0\left|A\subset V_\eta(B)~~\text{and}~~B\subset V_\eta(A)\right. \right\},
\end{align*}
where for any subset $A\subset M$, $V_\eta(A)$ is the tubular neighborhood of $A$ defined by $V_\eta(A):=\left\{x\in M \left| d(x,A)<\eta \right. \right\}$.

Let $(X,d_X)$ and $(Y,d_Y)$ be two metric spaces. Given a real number $\t>0$, a {\em $\t$-isometry} ({\em almost isometry}) is a map $F:(X,d_X)\ra (Y,d_Y)$ satisfying
\begin{enumerate}[(1)]
\item for any $x_1,x_2\in X$, $\left|d_Y\(F(x_1),F(x_2)\)-d_X(x_1,x_2)\right|<\t$;
\item for any point $y\in Y$, there exists a point $x\in X$ with $d_Y\(y,F(x)\)<\t.$
\end{enumerate}

Let $(M^n,g)$ be a closed, connected and oriented $n$-dimensional Riemannian manifold isometrically immersed by $\phi$ in $(\S^{n+1},can)$. In the sequel, we denote by $\nabla$ and $\-\nabla$ the gradient associated to $g$ and $can$ respectively. For any $p_0\in \S^{n+1}$, let $\exp_{p_0}$ be the exponential map at this point and $r(x):=d(p_0,x)$ be the distance function to this point. We denote by $(x_i)_{1\leq i \leq n+1}$ the normal coordinates of $\S^{n+1}$ centered at $p_0$. The position vector $X$ is defined by $X:=\sin r \-\nabla r$, and it is easy to observe that the normal coordinates of $X$ are $\frac{\sin r}{r}x_i$. We will use $\frac{\sin r}{r}x_i$ as test functions in the variational characterization of $\l_1(M)$ but the mean of these functions must be zero. For this purpose, there are two ways to choose a specific point $p_0\in\S^{n+1}$. One is the {\em center of mass} of $M$, which is given by Chavel \cite{Chavel2006} and Heintze \cite{Heintze1988}. Indeed, we assume that $\phi(M)$ is contained in an open hemisphere of $\S^{n+1}$. Let $Y$ be the vector field in this open hemisphere defined by
\begin{align*}
Y_q=\int_M \frac{\sin d(q,x)}{d(q,x)}\exp_q^{-1}(x)dv(x)\in T_q \S^{n+1},
\end{align*}
then by Brouwer's fixed point theorem, there exists a point $p_0$ in this open hemisphere such that $Y_{p_0}=0$. Hence for this $p_0$, we have $\int_M \frac{\sin r}{r}x_i dv=0$.
The other one is the {\em center of gravity} of $M$, which is introduced by Veeravalli \cite{Veeravalli2001}. It is defined as a critical point of the smooth function $$\mathcal{E}:\S^{n+1}\ra \R, \quad p\mapsto \int_M \left[1-\cos d(p,x)\right]dv(x).$$ The introduction of $\cos(\cdot)$ has a significant advantage: the distance function $d(p,\cdot)$ may be non-smooth at some points, but thanks to $\cos(\cdot)$, $\mathcal{E}$ is smooth on the whole manifold $\S^{n+1}$. Furthermore, by the compactness of $\S^{n+1}$, there exist at least two centers of gravity. If $p_0\in\S^{n+1}$ is a critical point of $\mathcal{E}$, then for any unit vector $u\in T_{p_0}\S^{n+1}$ we have
\begin{align*}
0=\langle \-\nabla \mathcal{E}(p_0),u \rangle=-\int_M \langle \-\nabla \cos r(x),u\rangle dv(x),
\end{align*}
which indicates that $\int_M \frac{\sin r}{r}x_i dv=0$. Therefore, a center of gravity of $M$ is also a center of mass of $M$. In the sequel, we assume that $\phi(M)$ contained in $\-B(p_0,R)$, where $p_0$ is a center of gravity of $M$ and radius $R<\frac{\pi}{2}$. We recall two lemmas shown by Heintze \cite{Heintze1988}.
\begin{lem}\label{Heintze-lem-1}
At any $x\in M$, we have
\begin{align}\label{generalized-equality-I}
\sum_{i=1}^{n+1}\left|\nabla \(\frac{\sin r}{r}x_i \)\right|^2=n-|X^\top|^2.
\end{align}
\end{lem}
\begin{lem}\label{Heintze-lem-2}
The vector field $X=\sin r \-\nabla r$ satisfies
\begin{align}\label{generalized-equality-II}
\divv (X^\top)=n\cos r-nH\langle X,\nu \rangle.
\end{align}
\end{lem}

We also recall an inequality given by Jorge and Xavier \cite{Jorge-Xavier1981}, which relates the extrinsic radius and the mean curvature. If $\phi(M)$ lies in a closed geodesic ball of radius $R<\frac{\pi}{2}$, then $\|H\|_\infty \geq \cot(R)=:\b$. By using the identity (\ref{generalized-equality-II}) and $X^\top=\sin r \nabla r$, we obtain
\begin{align*}
\frac{1}{|M|}\int_M(n-|X^\top|^2)dv =&\frac{1}{|M|}\int_M \(n\sin^2 r+nH\langle X,\nu \rangle\cos r\)dv \\
\leq &\frac{n}{|M|}\int_M \sin^2 rdv+\frac{n\|H\|_\infty}{|M|}\int_M \sin r\cos r dv \\
=    &n\|X\|_2^2+\frac{\|H\|_\infty}{|M|}\int_M \sin r \(\divv(X^\top)+nH\langle X,\nu \rangle\) dv \\
\leq &n(1+\|H\|_\infty^2) \|X\|_2^2-\frac{\|H\|_\infty}{|M|}\int_M \cot r |X^\top|^2 dv.
\end{align*}
Since $\cot(\cdot)$ is decreasing on $[0,\frac{\pi}{2})$, we get
\begin{align}\label{X-L2-norm-lower-bound}
1 \leq (1+\|H\|_\infty^2) \|X\|_2^2+\frac{1}{n}\(1-\b^2\)\|X^{\top}\|_2^2.
\end{align}
Since $p_0$ is a center of gravity of $M$, we may take $\{\frac{\sin r}{r}x_i\}_{1\leq i\leq n+1}$ as test functions. By the identity (\ref{generalized-equality-I}) and the variational characterization of the first nonzero eigenvalue, we get
\begin{align}\label{hemi-sphere-condition}
\l_1(M)\|X\|_2^2 \leq \frac{1}{|M|}\int_M (n-|X^\top|^2)dv \leq n(1+\|H\|_\infty^2)\|X\|_2^2-\b^2\|X^{\top}\|_2^2.
\end{align}

\section{An $L^2$-Approach}\label{L2approach}
We consider the composition of isometric immersions $M^n \overset{\phi}{\hookrightarrow} \S^{n+1}\overset{i}{\hookrightarrow} \R^{n+2}$, where $i$ is the the canonical embedding of $\S^{n+1}$ into $\R^{n+2}$. We denote by $\-{H}$ the mean curvature of $i\circ \phi$, then $|\-{H}|^2 =1+|H|^2$. Applying the Hoffman-Spruck Sobolev inequality \cite{Hoffman-Spruck1974} to the isometric immersion $i\circ \phi: M^n \ra  \R^{n+2}$, we immediately derive a Sobolev-type inequality for closed hypersurfaces in $\S^{n+1}$.
\begin{lem}\label{Sobolev-spherical}
Let $(M^n,g)(n\geq 2)$ be a closed hypersurface isometrically immersed in $\S^{n+1}$. For a nonnegative function $f\in C^1(M)$, we have
\begin{align}\label{Sobolev-spherical-inequality}
\(\int_M f^\frac{n}{n-1}dv\)^\frac{n-1}{n} \leq K(n)\int_M (|\nabla f|+f\sqrt{1+|H|^2})dv,
\end{align}
where $K(n)$ is a positive constant depending only on $n$.
\end{lem}

Now we give some $L^2$-estimates under the eigenvalue pinching condition. For simplicity, we denote by $(\Lambda_\e)$ the pinching condition $n(1+\|H\|_\infty^2)-\e\leq \l_1(M)$. We also abbreviate $\l_1=\l_1(M)$, $h=\sqrt{1+\|H\|_\infty^2}$. Moreover, we omit the volume element $dv$ in the integrals when there is no confusion. First, we have
\begin{lem}\label{Xt-L2-estimate}
If $(\Lambda_\e)$ holds, then
\begin{align}\label{L2-estimate-Xt}
\|X^\top\|_2^2 \leq \frac{\e}{\b^2}\|X\|_2^2.
\end{align}
\end{lem}
\begin{proof}
By (\ref{hemi-sphere-condition}) and ($\L_\e$), we have
\begin{align*}
nh^2-\e\leq nh^2-\b^2\frac{\|X^\top\|_2^2}{\|X\|_2^2},
\end{align*}
which gives the desired estimate.
\end{proof}
Since $\|X^\top\|_2^2\leq \|X\|_2^2$, it is natural to assume that $\e\leq \b^2$.

\begin{lem}\label{X-L2-estimate}
If $(\L_\e)$ holds with $\e\leq \b^2$, then
\begin{align}\label{L2-estimate-X}
\frac{n}{(n+1)h^2} \leq \|X\|_2^2 \leq \frac{1+\e}{h^2} \leq 1.
\end{align}
\end{lem}
\begin{proof}
Since $\e\leq \b^2 \leq \frac{n}{2}\b^2\leq \frac{nh^2}{2}$, then ($\L_\e$) implies that $\l_1\geq \frac{nh^2}{2}\geq \frac{n}{2}\geq 1$. Then by (\ref{hemi-sphere-condition}) we get
\begin{align*}
nh^2\|X\|_2^2 &\leq (\l_1+\e)\|X\|_2^2 \leq \l_1(1+\e)\|X\|_2^2\\
              &\leq (1+\e)(n-\|X^\top\|_2^2)\leq n(1+\e),
\end{align*}
which proves the right-hand side of (\ref{L2-estimate-X}). On the other hand, by (\ref{X-L2-norm-lower-bound}) and Lemma \ref{Xt-L2-estimate}, we have
\begin{align*}
1 \leq h^2 \|X\|_2^2+\frac{1}{n}\|X^\top\|_2^2 \leq h^2\(1+\frac{\e}{n\b^2}\)\|X\|_2^2 \leq \frac{(n+1)h^2}{n}\|X\|_2^2.
\end{align*}
\end{proof}

Let $Y=n\cos r H\nu-n\|H\|_\infty^2 X$.
\begin{lem}\label{Y-L2-estimate}
If $(\Lambda_\e)$ holds with $\e\leq \b^2$, then
\begin{align}\label{L2-estimate-Y}
\|Y\|_2^2 \leq 4n^2h^2\e\(1+\frac{1}{\b^2}\).
\end{align}
\end{lem}
\begin{proof}
By the identity (\ref{generalized-equality-II}), Lemmas \ref{Xt-L2-estimate} and \ref{X-L2-estimate}, we have
\begin{align*}
\|Y\|_2^2=& \frac{n^2}{|M|}\int_M H^2\cos^2 r-\frac{2n^2\|H\|_\infty^2}{|M|}\int_M H\langle X,\nu \rangle \cos r+n^2\|H\|_\infty^4
\|X\|_2^2 \\
\leq & -n^2\|H\|_\infty^2(1-\|X\|_2^2)+2n\|H\|_\infty^2\|X^\top\|_2^2+n^2\|H\|_\infty^4\|X\|_2^2 \\
\leq &  n^2\|H\|_\infty^2 \e +4n^2\|H\|_\infty^2 \frac{\e}{\b^2} \\
\leq &  4n^2h^2\e\(1+\frac{1}{\b^2}\).
\end{align*}
\end{proof}

Let $W=|X|^{-\frac{1}{2}}\(X|X|+|X|\cos r H\nu-hX\)$.
\begin{lem}\label{W-L2-estimate}
If $(\Lambda_\e)$ holds with $\e\leq \b^2$, then
\begin{align}\label{L2-estimate-W}
\|W\|_2^2 \leq  4h\e\(1+\frac{1}{\b^2}\).
\end{align}
\end{lem}
\begin{proof}
By Lemmas \ref{Xt-L2-estimate} and \ref{X-L2-estimate}, a similar computation gives
\begin{align*}
\|W\|_2^2 \leq &\frac{1}{|M|}\int_M \left[|X||X+H\cos r\nu|^2-2h\langle X+H\cos r\nu,X\rangle+h^2|X|\right] \\
\leq &\frac{1}{|M|}\int_M \left[|X||X+H\cos r\nu|^2-2h\langle X+H\cos r\nu,X\rangle\right]+h^2\|X\|_2\\
\leq &2h^2\|X\|_2-2h+\frac{2h}{n}\|X^\top\|_2^2 \\
\leq &4h\e\(1+\frac{1}{\b^2}\).
\end{align*}
\end{proof}

\begin{lem}\label{X-L-infty-estimate}
Let $A>0$ be a real number such that $|M|^\frac{1}{n}(1+\|H\|_2) \leq A$. If $(\Lambda_\e)$ holds with $\e\leq \min\(1,\b^2\)$, then there exists a positive constant $C_1=C_1(n,A)$ such that
\begin{align}\label{Linfty-estimate-X}
\|X\|_\infty \leq C_1\|X\|_2.
\end{align}
\end{lem}
\begin{proof}
Since $\varphi:=|X|=\sin r$, we have $|d\varphi^{2\a}| \leq 2\a\varphi^{2\a-1}\cos r \leq 2\a\varphi^{2\a-1}$. Hence using Lemma \ref{Sobolev-spherical}, we get for any $\a\geq 1$ and $f=\varphi^{2\a}$,
\begin{align*}
\|\varphi\|_\frac{2\a n}{n-1}^{2\a} \leq K(n)|M|^\frac{1}{n}(2\a+h\|\varphi\|_\infty)\|\varphi\|_{2\a-1}^{2\a-1}.
\end{align*}
By Lemma \ref{X-L2-estimate}, we have
$$
\|\varphi\|_\infty \geq \|\varphi\|_2 \geq \frac{1}{h}\sqrt{\frac{n}{n+1}}.
$$
Then taking $\nu=\frac{n}{n-1}$, $\a=\frac{a_p+1}{2}$, where $a_{p+1}=(a_p+1)\nu$ and $a_0=2$, we get
\begin{align*}
\(\frac{\|\varphi\|_{a_{p+1}}}{\|\varphi\|_\infty}\)^{\frac{a_{p+1}}{\nu^{p+1}}} \leq &\left[K(n)|M|^\frac{1}{n}\(\frac{a_p+1}{\|\varphi\|_\infty}+h\)\right]^\frac{1}{\nu^p}\(\frac{\|\varphi\|_{a_{p}}}{\|\varphi\|_\infty}\)^{\frac{a_{p}}{\nu^p}}\\
\leq &\left[5K(n)|M|^\frac{1}{n}a_p h\right]^\frac{1}{\nu^p}\(\frac{\|\varphi\|_{a_{p}}}{\|\varphi\|_\infty}\)^{\frac{a_{p}}{\nu^p}}\\
\leq &\(\prod_{i=1}^{p}a_i^{\frac{1}{\nu^i}}\)\left[5K(n)|M|^\frac{1}{n}h\right]^{n\(1-\frac{1}{\nu^{p+1}}\)}\(\frac{\|\varphi\|_{a_{0}}}{\|\varphi\|_\infty}\)^{a_{0}}.
\end{align*}
By ($\L_\e$) with $\e\leq 1$ and (\ref{Reilly-inequality-spherical}), we see
\begin{align*}
|M|^\frac{1}{n}h \leq |M|^\frac{1}{n}\sqrt{(1+\e)(1+\|H\|_2^2)} \leq \sqrt{2}|M|^\frac{1}{n}(1+\|H\|_2) \leq \sqrt{2}A.
\end{align*}
Observing that $\frac{a_p}{\nu^p}$ converges to $a_0+n$ and $a_0=2$, we get
\begin{align*}
\|\varphi\|_\infty \leq C(n)A^\frac{n}{2}\|\varphi\|_2=:C_1(n,A)\|\varphi\|_2.
\end{align*}
\end{proof}

Let us introduce the function $\psi:=|X|^\frac{1}{2}||X|-\frac{1}{h}|=|X|^\frac{1}{2}|X-\frac{1}{h}\frac{X}{|X|}|$. First, we give an $L^1$-estimate of $\psi$.
\begin{lem}\label{psi-L1-estimate}
If $(\Lambda_\e)$ holds with $\e\leq \b^2$, then
\begin{align}\label{L1-estimate-psi}
\|\psi\|_1 \leq \frac{5\e^\frac{1}{2}}{h^\frac{3}{2}}\(1+\frac{1}{\b^2}\)^\frac{1}{2}.
\end{align}
\end{lem}
\begin{proof}
First we have
\begin{align*}
\psi=&|X|^\frac{1}{2}\left|\frac{1}{h^2}(h^2X-X-H\cos r\nu)+\frac{1}{h^2}\(X+H\cos r\nu-h\frac{X}{|X|}\) \right| \\
\leq &\frac{|X|^\frac{1}{2}}{nh^2}|Y|+\frac{1}{h^2}|W|
\end{align*}
Then by the H\"older inequality, Lemmas \ref{Y-L2-estimate} and \ref{W-L2-estimate}, we get
\begin{align*}
\|\psi\|_1 \leq &\frac{1}{h^2}\(\frac{1}{n}\|X\|_2^\frac{1}{2}\|Y\|_2+\|W\|_2\)\\
\leq &\frac{1}{h^2}\left\{\frac{1}{n}\(\frac{2}{h^2}\)^\frac{1}{4}\left[4n^2h^2\e\(1+\frac{1}{\b^2}\)\right]^\frac{1}{2}+\left[4h\e\(1+\frac{1}{\b^2}\)\right]^\frac{1}{2}\right\}\\
\leq &\frac{5\e^\frac{1}{2}}{h^\frac{3}{2}}\(1+\frac{1}{\b^2}\)^\frac{1}{2}.
\end{align*}
\end{proof}

Now we give an $L^\infty$-estimate of $\psi$.
\begin{lem}\label{psi-L-infty-estimate}
Let $A>0$ be a real number such that $|M|^\frac{1}{n}(1+\|H\|_2) \leq A$. If $(\Lambda_\e)$ holds with $\e\leq \min\left\{1,\b^2 \right\}$, then there exists a positive constant $C_2=C_2(n,\b,A)$ such that
\begin{align}\label{Linfty-estimate-psi}
\|\psi\|_\infty \leq \frac{C_2\e^\frac{1}{2(2n+1)}}{h^\frac{3}{2}}.
\end{align}
Furthermore, $C_2(n,\b,A)\ra \infty$ as $\b\ra 0$.
\end{lem}

\begin{proof}
Let $\a\geq 1$, then by Lemma \ref{X-L-infty-estimate} we have
\begin{align*}
|d\psi^{2 \a}|=&\a \psi^{2\a-2}|d(\psi^2)|=\a \psi^{2\a-2}\left||X|-\frac{1}{h}\right|\left|3|X|-\frac{1}{h} \right||d|X|| \\
          \leq &3\a \psi^{2\a-2}\(\|X\|_\infty+\frac{1}{h}\)^2 \cos r \\
          \leq &3\a \(\frac{\sqrt{2}C_1+1}{h}\)^2 \psi^{2\a-2}.
\end{align*}
A direct calculation gives
$$
\|\psi\|_\infty^2 \leq \|X\|_\infty \(\|X\|_\infty+\frac{1}{h}\)^2 \leq \frac{\sqrt{2}C_1(\sqrt{2}C_1+1)^2}{h^3}.
$$
By using Lemma \ref{Sobolev-spherical}, we obtain for any $\a\geq 1$
\begin{align*}
\|\psi\|_{\frac{2\a n}{n-1}}^{2 \a} \leq &K(n)|M|^{\frac{1}{n}}\left[ 3\a\frac{(\sqrt{2}C_1+1)^2}{h^2}+\frac{\sqrt{2}C_1(\sqrt{2}C_1+1)^2}{h^2} \right] \|\psi\|_{2\a-2}^{2\a-2}\\
\leq &3\a K(n) |M|^{\frac{1}{n}}h \frac{\(\sqrt{2}C_1+1\)^3}{h^3}\|\psi\|_{2\a-2}^{2\a-2}\\
\leq &3\a K(n) \sqrt{2} A \frac{\(\sqrt{2}C_1+1\)^3}{h^3}\|\psi\|_{2\a-2}^{2\a-2}.
\end{align*}
Now we put $a_{p+1}=(a_p+2)\nu$, with $\nu=\frac{n}{n-1}$, $a_0=1$ and $\a=\frac{a_p+2}{2}$, we get
\begin{align*}
\|\psi\|_{a_{p+1}}^{\frac{a_{p+1}}{\nu^{p+1}}} \leq & \left[3\sqrt{2}K(n)a_p A \frac{\(\sqrt{2}C_1+1\)^3}{h^3}\right]^{\frac{1}{\nu^{p}}}\|\psi\|_{a_p}^{\frac{a_p}{\nu^p}} \\
\leq &\(\prod_{i=0}^{p}a_i^{\frac{1}{\nu^i}}\) \left[3\sqrt{2}K(n)A \frac{\(\sqrt{2}C_1+1\)^3}{h^3}\right]^{n\(1-\frac{1}{\nu^{p+1}}\)}\|\psi\|_{a_0}^{a_0}.
\end{align*}
Note that $\frac{a_p}{\nu^p}$ converges to $a_0+2n$. By Lemma \ref{psi-L1-estimate}, we have
\begin{align*}
\|\psi\|_\infty \leq \left[\frac{C(n,A)}{h^3}\right]^\frac{n}{2n+1}\left[\frac{5\e^\frac{1}{2}}{h^\frac{3}{2}}\(1+\frac{1}{\b^2}\)^\frac{1}{2}\right]^\frac{1}{2n+1}
=: \frac{C_2(n,\b,A)}{h^\frac{3}{2}}\e^\frac{1}{2(2n+1)}.
\end{align*}
\end{proof}

\begin{lem}\label{radius-estimate}
Let $A>0$ be a real number such that $|M|^\frac{1}{n}(1+\|H\|_2) \leq A$. There exists a positive constant $\e_0=\e_0(n,\b,A)$ such that if $(\Lambda_\e)$ holds with $\e\leq \e_0$, then there exists a positive constant $C_3=C_3(n,\b,A)$ satisfying
\begin{align}\label{X-pinching}
\left| |X|-\frac{1}{h}\right| \leq \frac{C_3\e^\frac{1}{2(2n+1)}}{h}, \quad \text{and} \quad \left| r-R_0\right| \leq C_3\e^\frac{1}{2(2n+1)} R_0,
\end{align}
where $R_0=\arcsin\frac{1}{\sqrt{1+\|H\|_\infty^2}}$. Furthermore, $C_3(n,\b,A)\ra \infty$ as $\b \ra 0$.
\end{lem}

\begin{proof}
Consider the function $f(t):=t\(t-\frac{1}{h}\)^2$, which is increasing on $\left[0,\frac{1}{3h}\right]$ and $\left[\frac{1}{h},+\infty\right)$, and decreasing on $\left[\frac{1}{3h},\frac{1}{h}\right]$. We take $\e_0=\e_0(n,\b,A)>0$ sufficiently small such that
\begin{align*}
9C_2^2\e_0^\frac{1}{2n+1} \leq 1, \quad \text{and} \quad \e_0 \leq \min \(1,\b^2\).
\end{align*}
Then by using Lemma \ref{psi-L-infty-estimate} and $(\L_\e)$ with $\e\leq \e_0$, we see
$$
\|f(|X|)\|_\infty=\|\psi\|_\infty^2 \leq \frac{1}{9h^3}<\frac{4}{27h^3}=f\(\frac{1}{3h}\).
$$
By Lemma \ref{X-L2-estimate}, we have $\|X\|_2^2\geq \frac{n}{(n+1)h^2} \geq \frac{1}{4h^2}$. Hence there exists $x_0\in M$ satisfying $|X_{x_0}|\geq \frac{1}{2h}>\frac{1}{3h}$. By the connectedness of $M$, it follows that $|X|>\frac{1}{3h}$ on $M$. This implies
\begin{align*}
\left| |X|-\frac{1}{h} \right| \leq \frac{\sqrt{3}C_2(n,\b,A)}{h}\e^\frac{1}{2(2n+1)}=:\frac{C_3(n,\b,A)}{h}\e^\frac{1}{2(2n+1)}.
\end{align*}
Since $|X|\leq \sin R=\frac{1}{\sqrt{1+\b^2}}$ and $R_0=\arcsin\frac{1}{h}\geq \(1+\frac{1}{\b^2}\)^\frac{1}{2}\frac{1}{h}$, we get
\begin{align*}
\left| r-R_0 \right| \leq \(1+\frac{1}{\b^2}\)^\frac{1}{2}\frac{C_3(n,\b,A)}{h}\e^\frac{1}{2(2n+1)}\leq C_3(n,\b,A)\e^\frac{1}{2(2n+1)}R_0.
\end{align*}
\end{proof}

\section{Proof of the Main Theorem}
Let us consider the map
\begin{align*}
F:~~&M \ra     S\(p_0,R_0\), \\
  ~~&x \mapsto \exp_{p_0}\(R_0\frac{Y}{|Y|}\),
\end{align*}
where $Y:=\exp_{p_0}^{-1}(x)$ and $p_0$ is a center of gravity of $M$.

\begin{lem}\label{differential-of-F}
Let $u\in U_x M$ and $v=u-\langle u,\nabla r \rangle \-{\nabla}r$. We have
\begin{align}\label{Df-estimate}
|dF_x(u)|^2 =\frac{1}{h^2\sin^2 r}|v|^2.
\end{align}
Furthermore, we have
\begin{align}\label{Df-two-side-estimate}
\frac{1-\|\nabla r\|_\infty^2}{h^2 \sin^2 r}\leq |dF_x(u)|^2 \leq \frac{1}{h^2 \sin^2 r}.
\end{align}
\end{lem}
\begin{proof}
The proof of equality (\ref{Df-estimate}) can be found in \cite{Grosjean-Roth2012}. By the construction of $v$, we immediately obtain (\ref{Df-two-side-estimate}).
\end{proof}

\begin{lem}\label{Xt-L-infty-estimate}
Let $q>n$ and $A>0$ be some real numbers such that $|M|^\frac{1}{n}(1+\|B\|_q) \leq A$. If $(\Lambda_\e)$ holds with $\e\leq \min\(1,\b^2\)$, then there exists a positive constant $C_4=C_4(q,n,\b,A)$ such that
\begin{align}\label{L-infty-estimate-Xt}
\|X^\top\|_\infty \leq \frac{C_4\e^\frac{1}{2(1+\g)}}{h},
\end{align}
where $\g=\frac{qn}{q-n}$. Moreover, $C_4(q,n,\b,A)\ra \infty$ as $q\ra n$ or $\b \ra 0$.
\end{lem}
\begin{proof}
Put $\chi=|X^\top|$. For any $\a\geq 1$, we have
$$
|d\chi^{2\a}|=2\a \chi^{2\a-1}\(\cos r|\nabla r|^2+\sin r\left|d|\nabla r|\right|\).
$$
Let us estimate $|d|\nabla r||$ at a point $x\in M$. For this, let $\left\{ e_i \right\}_{1\leq i \leq n}$ be an orthonormal basis of $T_x M$. We have
\begin{align*}
|d|\nabla r||^2=&\frac{1}{4|\nabla r|^2}|d\langle \-{\nabla}r,\nu \rangle^2|^2=\frac{\langle \-\nabla r,\nu \rangle^2}{|\nabla r|^2} \sum_{i=1}^{n}\(\-{\nabla}dr(e_i,\nu)+B(e_i,\-{\nabla}r)\)^2 \\
\leq &\frac{2}{|\nabla r|^2}\left[\sum_{i=1}^{n}\-{\nabla}dr(e_i,\nu)^2+|B|^2|\nabla r|^2 \right].
\end{align*}
Now $\sum_{i=1}^n \-{\nabla}dr(e_i,\nu)^2 \leq |\-\nabla dr|^2 \leq \sum_{i=1}^{n+1}\-{\nabla}dr(u_i,u_i)^2$, where $\{u_i\}_{1\leq i\leq n+1}$ is an orthonormal basis which diagonalizes $\-{\nabla}dr$. From the comparison theorem for the Hessian of distance function (see \cite{Petersen2006}, Theorem $27$), we deduce that
\begin{align*}
\sum_{i=1}^{n+1}\-{\nabla}dr(u_i,u_i)^2 \leq \(\frac{\cos r}{\sin r}\)^2 \sum_{i=1}^{n+1}|u_i-\langle u_i,\-{\nabla} r \rangle \-{\nabla} r|^2=n\(\frac{\cos r}{\sin r}\)^2.
\end{align*}
It follows that
\begin{align*}
|d|\nabla r||^2 \leq \frac{2n}{|\nabla r|^2}\(\frac{\cos r}{\sin r}\)^2+2|B|^2,
\end{align*}
and
\begin{align*}
|d\chi^{2\a}|\leq & 2\a \chi^{2\a-1}C(n)\left[\cos r+\frac{\cos r}{|\nabla r|}+\sin r|B| \right] \\
             \leq & 4\a \chi^{2\a-2}C(n)\|X\|_\infty\(1+|B|\|X\|_\infty\).
\end{align*}
A straightforward calculation gives
\begin{align*}
|M|^\frac{1}{n}h \leq \sqrt{2}|M|^\frac{1}{n}(1+\|H\|_2)\leq \sqrt{2}|M|^\frac{1}{n}(1+\|B\|_q).
\end{align*}
By using Lemma \ref{Sobolev-spherical}, we get for $\a\geq 1$
\begin{align*}
\|\chi\|_{\frac{2\a n}{n-1}}^{2\a}\leq &K(n)|M|^\frac{1}{n} \left[ 4\a C(n)\|X\|_\infty \|\chi\|_{2\a-2}^{2\a-2}\right. \\
&+4\a C(n)\|X\|_\infty^2\|B\|_q\|\chi\|_\frac{(2\a-2)q}{q-1}^{2\a-2}\left. +h\|\chi\|_{2\a}^{2\a}\right] \\
\leq &C'(n)\a \|X\|_\infty^2\left[|M|^\frac{1}{n}h+|M|^\frac{1}{n}\|B\|_q \right]\|\chi\|_\frac{(2\a-2)q}{q-1}^{2\a-2}  \\
\leq &C(n,A)\a\|X\|_\infty^2 \|\chi\|_\frac{(2\a-2)q}{q-1}^{2\a-2}.
\end{align*}
We take $\nu:=\frac{n(q-1)}{(n-1)q}$, $a_{p+1}:=a_p\nu+\frac{2n}{n-1}$, $a_0=2$ and $\a:=\frac{q-1}{2q}a_p+1$. Then $a_{p+1}=\frac{2\a n}{n-1}$ and
\begin{align*}
\|\chi\|_{a_{p+1}}^{\frac{a_{p+1}}{\nu^{p+1}}} \leq &\left[C(n,A)\|X\|_\infty^2 a_{p+1}\right]^\frac{n}{(n-1)\nu^{p+1}}\|\chi\|_{a_{p}}^{\frac{a_{p}}{\nu^{p}}} \\
\leq &\(\prod_{i=1}^{p+1}a_i^\frac{1}{\nu^i}\)^\frac{n}{n-1}\left[C(n,A)\|X\|_\infty^2\right]^{\frac{n}{n-1}\sum_{i=1}^{p+1}\frac{1}{\nu^{i}}}\|\chi\|_{a_{p}}^{\frac{a_{p}}{\nu^{p}}}.
\end{align*}
Note that $\frac{a_p}{\nu^p}$ converges to $a_0+\frac{2nq}{q-n}$. By Lemmas \ref{Xt-L2-estimate} and \ref{X-L-infty-estimate} we have
\begin{align*}
\|\chi\|_\infty \leq& C(q,n,A)\|X\|_\infty^\frac{\g}{1+\g}\|\chi\|_2^\frac{1}{1+\g} \\
\leq& C(q,n,A)\|X\|_\infty^\frac{\g}{1+\g}\(\frac{\e}{\b^2}\|X\|_2^2\)^\frac{1}{2(1+\g)} \\ =:&\frac{C_4(q,n,\b,A)}{h}\e^\frac{1}{2(1+\g)}.
\end{align*}
where $\g=\frac{qn}{q-n}$.
\end{proof}

\begin{proof}[Proof of Theorem \ref{main-theorem}]
Let $\e_0=\e_0(q,n,\b,A)>0$ such that
$$\max\left\{C_3\e_0^{\frac{1}{2(2n+1)}},C_4^2\e_0^\frac{1}{1+\g} \right\}\leq \frac{1}{16},\quad \text{and} \quad \e_0\leq \min\(1,\b^2\),$$
where $C_3=C_3(n,\b,A)$ and $C_4=C_4(q,n,\b,A)$ are the positive constants in Lemmas \ref{radius-estimate} and \ref{Xt-L-infty-estimate}.
If $(\Lambda_\e)$ holds with $\e\leq \e_0$, then we have
\begin{align*}
\frac{1}{h^2\sin^2 r}-1 \leq \frac{1}{\(1-C_3\e^\frac{1}{2(2n+1)}\)^2}-1\leq 8C_3\e^\frac{1}{2(2n+1)}.
\end{align*}
On the other hand, by Lemmas \ref{radius-estimate} and \ref{Xt-L-infty-estimate}, we get
\begin{align*}
\|\nabla r\|_\infty^2 \leq \frac{h^2\|X^\top\|_\infty^2}{\(1-C_3\e^\frac{1}{2(2n+1)}\)^2}
\leq 4C_4^2 \e^\frac{1}{1+\g},
\end{align*}
which gives
\begin{align*}
1-\frac{1-\|\nabla r\|_\infty^2}{h^2\sin^2 r} \leq 1-\frac{1-4C_4\e^\frac{1}{1+\g}}{(1+C_3\e^\frac{1}{2(2n+1)})^2}
\leq 3C_3\e^\frac{1}{2(2n+1)}+4C_4^2\e^\frac{1}{1+\g}.
\end{align*}
A direct computation shows
\begin{align*}
\left| |dF_x(u)|^2-1 \right| \leq 8\max \left\{C_3\e^\frac{1}{2(2n+1)},C_4^2\e^\frac{1}{1+\g} \right\}\leq \frac{1}{2}.
\end{align*}
Hence $F$ is a local diffeomorphism. Since $S\(p_0,R_0\)$ is simply connected for $n\geq 2$, then $F$ is a diffeomorphism.
Therefore, $Y$ is a bijection and $\phi$ is an embedding. By Lemmas \ref{radius-estimate} and \ref{Xt-L-infty-estimate}, we obtain
\begin{align*}
\min_{M}|X|\geq \frac{1-C_3\e^\frac{1}{2(2n+1)}}{h}> \frac{15}{16h},\quad \text{and} \quad
\max_{M}|X^\top| \leq \frac{C_4 \e^\frac{1}{2(1+\g)}}{h} < \frac{1}{4h},
\end{align*}
which implies that $\min_M |X^\bot|>0$. By the connectedness of $M$, $\langle X,\nu \rangle>0$ on $M$, which is equivalent to $\phi(M)$ is starshaped with respect to $p_0$. It remains to prove (\ref{almost-isometry}). A straightforward calculation gives
\begin{align*}
     &\left| d_2(F(x_1),F(x_2))-d_1(x_1,x_2)\right| \\
\leq & \max_{x\in M,u\in U_x M}\left| |dF_x(u)|^2-1 \right| d_1(x_1,x_2) \\
\leq & \sqrt{2}\max \left\{C_3(n,\b,A)\e^\frac{1}{2(2n+1)},C_4(q,n,\b,A)^2\e^\frac{1}{1+\g} \right\} \Diam\(S(p_0,R_0)\) \\
\leq & \frac{\sqrt{2}\pi}{h}\max \left\{C_3(n,\b,A)\e^\frac{1}{2(2n+1)},C_4(q,n,\b,A)^2\e^\frac{1}{1+\g} \right\}.
\end{align*}
This completes the proof of Theorem \ref{main-theorem}.
\end{proof}

\section{The Optimality of Radius Condition}\label{example}
In this section, we construct an example to show that the radius condition in Theorem \ref{main-theorem} is optimal.
In \cite{Muto-Ohnita-Urakawa1984}, Muto, Ohnita and Urakawa proved that for the great sphere $\S^n$ and the generalized Clifford torus $\S^{p}\(\sqrt{\frac{p}{n}}\)\times \S^{q}\(\sqrt{\frac{q}{n}}\)$ $(p+q=n)$ of $\S^{n+1}$, the first eigenvalue $\l_1$ is just $n$. We take $M$ be one of the generalized Clifford tori, and claim that every point of $\S^{n+1}$ can be chosen as a center of gravity of $M$. Since the radius of $M$ is exactly $\frac{\pi}{2}$, we may choose $p_0\in\S^{n+1}$ as a center of gravity of $M$ such that $\phi(M)\subset \-B(p_0,\frac{\pi}{2})$. However, $M$ is not homeomorphic to $\S^n$.
Now we prove the claim. It is equivalent to show that $\mathcal{E}$ is constant on $\S^{n+1}$. We fix $p\in \left\{1,\cdots,[\frac{n}{2}]\right\}$ and set $q=n-p$.
$$M=\S^{p}(r_1) \times \S^{q}(r_2)  \overset{\phi}{\hookrightarrow} \S^{n+1} \overset{i}\hookrightarrow \R^{n+2},$$
where $r_1=\sqrt{\frac{p}{n}}$ and $r_2=\sqrt{\frac{q}{n}}$.

By the spherical coordinates $(\t_1,\cdots,\t_p,\b_1,\cdots,\b_q)$, we have
$$
\left\{\begin{aligned}
x_1=&r_1\sin \t_1 \cos \t_2 \cdots \cos \t_{p-1} \cos \t_p, \\
x_2=&r_1\sin \t_1 \cos \t_2 \cdots \cos \t_{p-1} \sin \t_p, \\
&\quad \quad\vdots \\
x_{p}=&r_1\sin \t_1 \sin \t_2, \\
x_{p+1}=&r_1\cos \t_1,\\
x_{p+2}=&r_2\sin \b_1 \cos \b_2 \cdots \cos \b_{q-1} \cos \b_q, \\
x_{p+3}=&r_2\sin \b_1 \cos \b_2 \cdots \cos \b_{q-1} \sin \b_q, \\
&\quad \quad\vdots \\
x_{p+q+1}=&r_2\sin \b_1 \sin \b_2,  \\
x_{p+q+2}=&r_2\cos \b_1,
\end{aligned}
\right.
$$
where $\t_1,\b_1\in [0,\pi)$ and $\t_2,\cdots,\t_p,\b_2,\cdots,\b_q\in [0,2\pi)$. A direct computation gives
\begin{align*}
\sqrt{\det{g}}=r_1^p r_2^q \(\sin \t_1\)^{p-1}\(\sin \b_1\)^{q-1} \prod_{j=2}^{p-1}\(\cos \t_j\)^{p-j} \prod_{k=2}^{q-1}\(\cos \b_k\)^{q-k}.
\end{align*}
For any point $p=\(p_1,\cdots, p_{n+2}\)\in \S^{n+1} \hookrightarrow \R^{n+2}$, we have
\begin{align*}
\mathcal{E}(p)=&\int_M \left[ 1-\cos d(p,x)\right] dv(x) =\int_M \(1-\sum_{i=1}^{n+2}p_i x_i\) dv(x) \\
=&\int_{[0,\pi)^2 \times [0,2\pi)^{n-2}}\(1-\sum_{i=1}^{n+2}p_i x_i \) \sqrt{\det g}d\t_1 d\b_1  d\t_2\cdots d\t_{p} d\b_2 \cdots d\b_q \\
=&|M|.
\end{align*}
Therefore, $\mathcal{E}$ is constant on $\S^{n+1}$.

\end{document}